\documentclass[preprint,12pt]{article}

\usepackage{amsmath,amsthm,amssymb}
\usepackage{hyperref}
\usepackage{mathrsfs}  
\usepackage{dsfont}  
\usepackage{enumerate}  

\usepackage{marvosym}  

\allowdisplaybreaks[4]

\newtheorem{definition}{Definition}[section]
\newtheorem{lemma}{Lemma}[section]
\newtheorem{theorem}{Theorem}[section]
\newtheorem{corollary}{Corollary}[section]
\newtheorem{proposition}{Proposition}[section]

\newtheorem{example}{Example}[section]
\newtheorem{remark}{Remark}[section]

\numberwithin{equation}{section}

\newcommand{\sref}[1]{Section~\ref{#1}}
\newcommand{\ssref}[1]{Subsection~\ref{#1}}
\newcommand{\tref}[1]{\textsl{Theorem~\ref{#1}}}
\newcommand{\lref}[1]{\textsl{Lemma~\ref{#1}}}
\newcommand{\cref}[1]{\textsl{Corollary~\ref{#1}}}
\newcommand{\pref}[1]{\textsl{Proposition~\ref{#1}}}
\newcommand{\rref}[1]{\textsl{Remark~\ref{#1}}}
\newcommand{\eref}[1]{\textsc{(\ref{#1})}}

\newcommand{\norm}[1]{\left\Vert#1\right\Vert}
\newcommand{\set}[1]{\left\{#1\right\}}

\newcommand{\R}{\mathds R}

\newcommand{\ol}{\overline}
\newcommand{\ul}{\underline}
\newcommand{\p}{\partial}
\newcommand{\dps}{\displaystyle}
\newcommand{\wt}{\widetilde}
\newcommand{\ra}{\rightarrow}
\newcommand{\Ra}{\Rightarrow}

\begin{document}

\title{On the exterior Dirichlet problem
for Hessian quotient equations\footnotemark[1]}


\author{Dongsheng Li \footnotemark[3]
\and Zhisu Li (\Letter)\footnotemark[2]~\footnotemark[3]}

\renewcommand{\thefootnote}{\fnsymbol{footnote}}

\footnotetext[1]{This research is supported by NSFC.11671316.}
\footnotetext[2]{Corresponding author.}
\footnotetext[3]{School of Mathematics and Statistics,
Xi'an Jiaotong University, Xi'an 710049, China;\\
Dongsheng Li: \href{mailto:lidsh@mail.xjtu.edu.cn}{lidsh@mail.xjtu.edu.cn};
Zhisu Li: \href{mailto:lizhisu@stu.xjtu.edu.cn}{lizhisu@stu.xjtu.edu.cn}.}

\date{}


\maketitle

\begin{abstract}
In this paper, we establish
the existence and uniqueness theorem
for solutions of the exterior Dirichlet problem
for Hessian quotient equations
with prescribed asymptotic behavior at infinity.
This extends the previous related results on
the Monge-Amp\`{e}re equations and on the Hessian equations,
and rearranges them in a systematic way.
Based on the Perron's method,
the main ingredient of this paper is
to construct some appropriate subsolutions
of the Hessian quotient equation,
which is realized by
introducing some new quantities about
the elementary symmetric functions
and using them to analyze the corresponding
ordinary differential equation related to the
generalized radially symmetric subsolutions
of the original equation.
\end{abstract}





\section{Introduction}

In this paper, we consider the Dirichlet problem for
the Hessian quotient equation
\begin{equation}\label{eqn.hqe}
\frac{\sigma_k(\lambda(D^2u))}{\sigma_l(\lambda(D^2u))}=1
\end{equation}
in the exterior domain $\R^n\setminus\ol{D}$,
where $D$ is a bounded domain in $\R^n$,
$n\geq3$, $0\leq l<k\leq n$,
$\lambda(D^2u)$ denotes the eigenvalue vector
$\lambda:=(\lambda_1,\lambda_2,...,\lambda_n)$
of the Hessian matrix $D^2u$ of the function $u$,
and
\[\sigma_0(\lambda)\equiv1
\quad\text{and}\quad
\sigma_j(\lambda):=\sum_{1\leq s_1<s_2<...<s_j\leq n}
\lambda_{s_1}\lambda_{s_2}...\lambda_{s_j}
~(\forall 1\leq j\leq n)\]
are the elementary symmetric functions
of the $n$-vector $\lambda$.
Note that when $l=0$,
\eref{eqn.hqe} is the Hessian equation
$\sigma_k(\lambda(D^2u))=1$;
when $l=0$, $k=1$,
it is the Poisson equation $\Delta u=1$,
a linear elliptic equation;
when $l=0$, $k=n$,
it is the famous Monge-Amp\`{e}re equation
$\det(D^2u)=1$;
and when $l=1$, $k=3$, $n=3$ or $4$,
it is the special Lagrangian equation
$\sigma_1(\lambda(D^2u))=\sigma_3(\lambda(D^2u))$
in three or four dimension
(in three dimension, this is $\det(D^2u)=\Delta u$ indeed)
which arises from the special Lagrangian geometry \cite{HL82}.

For linear elliptic equations of second order,
there have been much extensive studies
on the exterior Dirichlet problem,
see \cite{MS60} and the references therein.
For the Monge-Amp\`{e}re equation,
a classical theorem of J\"{o}rgens \cite{Jor54},
Calabi \cite{Cal58} and Pogorelov \cite{Pog72}
states that any convex classical solution of
$\det(D^2u)=1$ in $\R^n$ must be a quadratic polynomial.
Related results was also given by
\cite{CY86}, \cite{Caf95}, \cite{TW00} and \cite{JX01}.
Caffarelli and Li \cite{CL03} extended
the J\"{o}rgens-Calabi-Pogorelov theorem
to exterior domains.
They proved that if $u$ is a convex viscosity solution of
$\det(D^2u)=1$ in the exterior domain $\R^n\setminus\ol{D}$,
where $D$ is a bounded domain in $\R^n$, $n\geq3$,
then there exist
$A\in\R^{n\times n},b\in\R^n$ and $c\in\R$
such that
\begin{equation}\label{eqn.abc-n}
\limsup_{|x|\ra+\infty}|x|^{n-2}
\left|u(x)-\left(\frac{1}{2}x^{T}Ax+b^Tx+c\right)\right|<\infty.
\end{equation}
With such prescribed asymptotic behavior at infinity,
they also established an existence and uniqueness theorem for
solutions of the Dirichlet problem
of the Monge-Amp\`{e}re equation
in the exterior domain of $\R^n$, $n\geq3$.
See \cite{FMM99}, \cite{FMM00} or \cite{Del92}
for similar problems in two dimension.
Recently, J.-G. Bao, H.-G. Li and Y.-Y. Li \cite{BLL14}
extended the above existence and uniqueness theorem
of the exterior Dirichlet problem in \cite{CL03}
for the Monge-Amp\`{e}re equation
to the Hessian equation
$\sigma_k(\lambda(D^2u))=1$ with $2\leq k\leq n$
and with some appropriate
prescribed asymptotic behavior at infinity
which is modified from \eref{eqn.abc-n}.
Before them, for the special case that $A=c_0 I$
with $c_0:=({C_n^k})^{-1/k}$ and $C_n^k:={n!}/(k!(n-k)!)$,
the exterior Dirichlet problem for the Hessian equation
has been investigated by Dai and Bao in \cite{DB11}.
At the same time,
Dai \cite{Dai11} proved the existence theorem
of the exterior Dirichlet problem
for the Hessian quotient equation
\eref{eqn.hqe} with $k-l\geq3$,
and with the prescribed asymptotic behavior at infinity
of the special case that $A=c_\ast I$, that is,
\begin{equation}\label{eqn.cc-kl}
\limsup_{|x|\ra+\infty}|x|^{k-l-2}
\left|u(x)-\left(\frac{c_\ast}{2}|x|^2+c\right)\right|<\infty,
\end{equation}
where
\begin{equation}\label{eqn.cst}
c_\ast:=\left(\frac{C_n^l}{C_n^k}\right)^{\frac{1}{k-l}}
~\text{with}~C_n^i:=\frac{n!}{i!(n-i)!}
~\text{for}~i=k,l.
\end{equation}
As they pointed out in \cite{LB14} that the restriction
$k-l\geq3$ rules out an important example,
the special Lagrangian equation
$\det(D^2u)=\Delta u$ in three dimension.
Later, \cite{LD12} improve the result in \cite{Dai11}
for \eref{eqn.hqe} with $k-l\geq3$
to that for \eref{eqn.hqe} with $0\leq l<k\leq(n+1)/2$.
More recently, Li and Bao \cite{LB14}
established the existence theorem
of the exterior Dirichlet problem
for a class of fully nonlinear elliptic equations
related to the eigenvalues of the Hessian
which include the Monge-Amp\`{e}re equations,
Hessian equations, Hessian quotient equations
and the special Lagrangian equations
in dimension equal and lager than three,
but with the prescribed asymptotic behavior at infinity
only in the special case of \eref{eqn.abc-n} that $A=c^\ast I$
with $c^\ast$ some appropriate constant,
like \eref{eqn.cc-kl} and \eref{eqn.cst}.

\medskip

In this paper, we focus our attention on
the Hessian quotient equation \eref{eqn.hqe}
and establish the existence and uniqueness theorem
for the exterior Dirichlet problem of it
with prescribed asymptotic behavior at infinity
of the type similar to \eref{eqn.abc-n}.
This extends the previous corresponding results
on the Monge-Amp\`{e}re equations \cite{CL03}
and on the Hessian equations \cite{BLL14}
to Hessian quotient equations,
and also extends those results
on the Hessian quotient equations
in \cite{Dai11}, \cite{LD12} and \cite{LB14}
to be valid for general
prescribed asymptotic behavior condition at infinity.
Since we do not restrict ourselves to
the case $k-l\geq3$ or $0\leq l<k\leq(n+1)/2$ only,
our theorems also apply to the special Lagrangian equations
$\det(D^2u)=\Delta u$ in three dimension
and $\sigma_1(\lambda(D^2u))=\sigma_3(\lambda(D^2u))$
in four dimension.
Indeed, we will show in our forthcoming paper \cite{LL16}
that our method still works very well for
the special Lagrangian equations with higher dimension
and with general phase.

We would like to remark that,
for the interior Dirichlet problems
there have been much extensive studies,
see for example \cite{CIL92}, \cite{CNS85},
\cite{Ivo85}, \cite{Kry83}, \cite{Urb90},
\cite{Tru90} and \cite{Tru95};
see \cite{BCGJ03} and the references given there
for more on the Hessian quotient equations;
and for more on the special Lagrangian equations,
we refer the reader to \cite{HL82}, \cite{Fu98},
\cite{Yuan02}, \cite{CWY09}
and the references therein.

\medskip

For the reader's convenience, we give the following definitions
related to Hessian quotient equation
(see also \cite{CIL92}, \cite{CC95}, \cite{CNS85},
\cite{Tru90}, \cite{Tru95} and the references therein).

We say that a function $u\in C^2\left(\R^n\setminus\ol{D}\right)$
is \emph{$k$-convex},
if $\lambda(D^2u)\in\ol{\Gamma_k}$ in $\R^n\setminus\ol{D}$,
where $\Gamma_k$ is the connected component of
$\set{\lambda\in\R^n|\sigma_k(\lambda)>0}$
containing the \emph{positive cone}
\[\Gamma^+:=\set{\lambda\in\R^n|\lambda_i>0,\forall i=1,2,...,n}.\]
It is well known that
$\Gamma_k$ is an open convex symmetric cone
with its vertex at the origin
and that
\[\Gamma_k=\set{\lambda\in\R^n|\sigma_j(\lambda)>0,
\forall j=1,2,...,k},\]
which implies
\[\set{\lambda\in\R^n|\lambda_1+\lambda_2+...+\lambda_n>0}
=\Gamma_1\supset...\supset\Gamma_k\supset\Gamma_{k+1}
\supset...\supset\Gamma_n=\Gamma^+\]
with the first term $\Gamma_1$ the half space
and with the last term $\Gamma_n$
the positive cone $\Gamma^+$.
Furthermore, we also know that
\begin{equation}\label{eqn.pisjp}
\p_{\lambda_i}\sigma_{j}(\lambda)>0,
~\forall 1\leq i\leq n,
~\forall 1\leq j\leq k,
~\forall\lambda\in\Gamma_k,
~\forall 1\leq k\leq n
\end{equation}
(see \cite{CNS85} or \cite{Urb90} for more details).

Let $\Omega$ be an open domain in $\R^n$
and let $f\in C^0(\Omega)$ be nonnegative.
Suppose $0\leq l<k\leq n$.
A function $u\in C^0(\Omega)$ is said to be
a \emph{viscosity subsolution} of
\begin{equation}\label{eqn.hqe-f}
\frac{\sigma_k(\lambda(D^2u))}{\sigma_l(\lambda(D^2u))}=f
\quad\text{in}~\Omega
\end{equation}
\big{(}or say that $u$ satisfies
\[\frac{\sigma_k(\lambda(D^2u))}{\sigma_l(\lambda(D^2u))}\geq f
\quad\text{in}~\Omega\]
in the viscosity sense, similarly hereinafter\big{)},
if for any function $v\in C^2(\Omega)$
and any point $x^\ast\in\Omega$ satisfying
\[v(x)\geq u(x),~\forall x\in\Omega
\quad\text{and}\quad v(x^\ast)=u(x^\ast),\]
we have
\[\frac{\sigma_k(\lambda(D^2v))}{\sigma_l(\lambda(D^2v))}\geq f
\quad\text{in}~\Omega.\]
A function $u\in C^0(\Omega)$ is said to be
a \emph{viscosity supersolution} of \eref{eqn.hqe-f}
if for any \textit{$k$-convex function} $v\in C^2(\Omega)$
and any point $x^\ast\in\Omega$ satisfying
\[v(x)\leq u(x),~\forall x\in\Omega
\quad\text{and}\quad v(x^\ast)=u(x^\ast),\]
we have
\[\frac{\sigma_k(\lambda(D^2v))}{\sigma_l(\lambda(D^2v))}\leq f
\quad\text{in}~\Omega.\]
A function $u\in C^0(\Omega)$ is said to be a
\emph{viscosity solution} of \eref{eqn.hqe-f},
if it is both a viscosity subsolution
and a viscosity supersolution of \eref{eqn.hqe-f}.
A function $u\in C^0(\Omega)$ is said to be a
\emph{viscosity subsolution}
(respectively, \emph{supersolution, solution})
of \eref{eqn.hqe-f} and $u=\varphi$ on $\p\Omega$
with some $\varphi\in C^0(\p\Omega)$,
if $u$ is a viscosity subsolution
(respectively, supersolution, solution)
of \eref{eqn.hqe-f} and
$u\leq$(respectively, $\geq,=$)$\varphi$ on $\p\Omega$.

Note that in the definitions of viscosity solution above,
we have used the ellipticity of the Hessian quotient equations indeed.
For completeness and convenience, this will be proved
in the end of \ssref{subsec.prelem}.
See also \cite{CC95}, \cite{CIL92}, \cite{CNS85},
\cite{Urb90} and the references therein.

\bigskip

Define
\[\mathscr{A}_{k,l}:=\left\{A\in S(n)\big{|}\lambda(A)\in\Gamma^+,
\sigma_k(\lambda(A))=\sigma_l(\lambda(A))\right\}.\]
Note that there are plenty of elements in $\mathscr{A}_{k,l}$.
In fact, for any $A\in S(n)$ with $\lambda(A)\in\Gamma^+$, if we set \[\varrho:=\left(\frac{\sigma_k(\lambda(A))}{\sigma_l(\lambda(A))}\right)
^{-\frac{1}{k-l}},\]
we then have $\varrho A\in\mathscr{A}_{k,l}$.
Let
\[\mathscr{\wt A}_{k,l}
:=\set{A\in\mathscr{A}_{k,l}\big{|}m_{k,l}(\lambda(A))>2},\]
where $m_{k,l}(\lambda)$ is a quantity
which plays an important role in this paper.
We will give the specific definition of $m_{k,l}(\lambda)$
in \eref{eqn.mkla} in \ssref{subsec.xi},
and verify there that $\mathscr{\wt A}_{k,l}$
possesses the following fine properties.
\begin{proposition}\label{prop.wtakl}
Suppose $0\leq l<k\leq n$ and $n\geq3$.
\begin{enumerate}[\quad(1)]
\item
If $k-l\geq2$,
then
$\mathscr{\wt A}_{k,l}=\mathscr{A}_{k,l}$.

\item
$\mathscr{\wt A}_{n,0}=\mathscr{A}_{n,0}$
and $m_{n,0}\equiv n$.

\item
$c_\ast I\in\mathscr{\wt A}_{k,l}$
and $m_{k,l}(c_\ast(1,1,...,1))=n$,
where $c_\ast$ is the one defined in \eref{eqn.cst}.
\end{enumerate}
\end{proposition}

\medskip

The main result of this paper now can be stated as below.
\begin{theorem}\label{thm.hqe}
Let $D$ be a bounded strictly convex domain
in $\R^n$, $n\geq 3$, $\p D\in C^2$
and let $\varphi\in C^2(\p D)$.
Then for any given $A\in\mathscr{\wt A}_{k,l}$
with $0\leq l<k\leq n$, and
any given $b\in \R^n$,
there exists a constant $\tilde{c}$
depending only on $n,D,k,l,A,b$
and $\norm{\varphi}_{C^2(\p D)}$,
such that for every $c\geq\tilde{c}$,
there exists a unique viscosity solution
$u\in C^0(\R^n\setminus D)$ of
\begin{equation}\label{eqn.hqe-abc}
\left\{
\begin{aligned}
\dps \frac{\sigma_k(\lambda(D^2u))}{\sigma_l(\lambda(D^2u))}&=1
\quad\text{in}~\R^n\setminus\ol{D},\\
\dps u&=\varphi\quad\text{on}~\p D,\\
\dps \limsup_{|x|\ra+\infty}|x|^{m-2}
&\left|u(x)-\left(\frac{1}{2}x^{T}Ax+b^Tx+c\right)\right|<\infty,
\end{aligned}
\right.
\end{equation}
where $m\in(2,n]$ is a constant depending
only on $n,k,l$ and $\lambda(A)$,
which actually can be taken as $m_{k,l}(\lambda(A))$.
\end{theorem}

\begin{remark}
\begin{enumerate}[(1)]
\item
One can easily see that
\tref{thm.hqe} still holds
with $A\in\mathscr{\wt A}_{k,l}$
replaced by $A\in\mathscr{\wt A}^\ast_{k,l}$
and $\lambda(A)$ replaced by $\lambda(A^\ast)$,
where
\[\mathscr{\wt A}^\ast_{k,l}
:=\left\{A\in\R^{n\times n}\big{|}
\lambda(A^\ast)\in\Gamma^+,
\sigma_k(\lambda(A^\ast))
=\sigma_l(\lambda(A^\ast)),
m_{k,l}(\lambda(A^\ast))>2\right\}\]
and $A^\ast:=(A+A^T)/2$.
This is to say that the above theorem can
be adapted to a slightly more general form
by modifying the meaning of $\mathscr{\wt A}_{k,l}$.

\item
For the special cases that $l=0$
(i.e., the Hessian equation $\sigma_k(\lambda(D^2u))=1$)
and that $l=0$ and $k=n$
(i.e., the Monge-Amp\`{e}re equation $\det(D^2u)=1$),
in view of \pref{prop.wtakl}-\textsl{(2)},
our \tref{thm.hqe} recovers the corresponding results
\cite[Theorem 1.1]{BLL14} and \cite[Theorem 1.5]{CL03},
respectively.

\item
For $A=c_\ast I$ with $c_\ast$ defined in \eref{eqn.cst},
by \pref{prop.wtakl}-\textsl{(1),(3)},
our results improve those in \cite{Dai11} and \cite{LD12}.
Indeed, by \pref{prop.wtakl}-\textsl{(3)},
the main results in \cite{Dai11}, \cite{LD12}
and those parts concerning
the Hessian quotient equations in \cite{LB14}
can all be recovered by \tref{thm.hqe} as special cases.
Furthermore, our results also apply to
the special Lagrangian equation
$\det(D^2u)=\Delta u$ in three dimension
(respectively,
$\sigma_1(\lambda(D^2u))=\sigma_3(\lambda(D^2u))$
in four dimension),
not only for $A=\sqrt{3}I$ (respectively, $A=I$), but also for
any $A\in\mathscr{A}_{3,1}$.
\qed
\end{enumerate}
\end{remark}

\medskip

The paper is organized as follows.
In \sref{sec.prel},
after giving some basic notations in \ssref{subsec.nott},
we introduce the definitions of
$\Xi_k,\ul\xi_k,\ol\xi_k$ and $m_{k,l}$,
and investigate their properties in \ssref{subsec.xi}.
Then we collect in \ssref{subsec.prelem}
some preliminary lemmas
which will be used in this paper.
\sref{sec.pmaint} is devoted to
the proof of the main theorem (\tref{thm.hqe}).
To do this, we start in \ssref{subsec.csubsol}
to construct some appropriate subsolutions of
the Hessian quotient equation \eref{eqn.hqe},
by taking advantages of the properties of
$\Xi_k,\ul\xi_k,\ol\xi_k$ and $m_{k,l}$
explored in \ssref{subsec.xi}.
Then in \ssref{subsec.pmaint},
after reducing \tref{thm.hqe} to
\lref{lem.hqe} by simplification and normalization,
we prove \lref{lem.hqe}
by applying the Perron's method to
the subsolutions we constructed in \ssref{subsec.csubsol}.

\section{Preliminary}\label{sec.prel}

\subsection{Notation}\label{subsec.nott}
\quad

In this paper,
$S(n)$ denotes the linear space of symmetric $n\times n$ real matrices,
and $I$ denotes the identity matrix.

For any $M\in S(n)$,
if $m_1,m_2,...,m_n$ are the eigenvalues of $M$
(usually, the assumption $m_1\leq m_2\leq...\leq m_n$
is added for convenience),
we will denote this fact briefly by $\lambda(M)=(m_1,m_2,...,m_n)$
and call $\lambda(M)$ the eigenvalue vector of $M$.

For $A\in S(n)$ and $\rho>0$, we denote by
\[
E_\rho:=\left\{x\in\R^n\big{|}x^TAx<\rho^2\right\}
=\left\{x\in\R^n\big{|}r_A(x)<\rho\right\}
\]
the ellipsoid of size $\rho$ with respect to $A$,
where we set $r_A(x):=\sqrt{x^TAx}$.

\medskip

For any $p\in\R^n$, we write
\[\sigma_k(p):=\sum_{1\leq s_1<s_2<...<s_k\leq n}
p_{s_1}p_{s_2}...p_{s_k}\quad(\forall 1\leq k\leq n)\]
as the $k$-th elementary symmetric function of $p$.
Meanwhile, we will adopt the conventions that
$\sigma_{-1}(p)\equiv0$, $\sigma_0(p)\equiv1$
and $\sigma_k(p)\equiv0$, $\forall k\geq n+1$;
and we will also define
\[\sigma_{k;i}(p):=\left(\sigma_{k}(\lambda)
\big{|}_{\lambda_i=0}\right)\Big{|}_{\lambda=p}
=\sigma_{k}\left(p_1,p_2,...,\widehat{p_i},...,p_n\right)\]
for any $-1\leq k\leq n$ and any $1\leq i\leq n$,
and similarly
\[\sigma_{k;i,j}(p):=\left(\sigma_{k}(\lambda)
\big{|}_{\lambda_i=\lambda_j=0}\right)\Big{|}_{\lambda=p}
=\sigma_{k}\left(p_1,p_2,...,\widehat{p_i},
...,\widehat{p_j},...,p_n\right)\]
for any $-1\leq k\leq n$
and any $1\leq i,j\leq n$, $i\neq j$, for convenience.

\subsection{Definitions and properties of
$\Xi_k,\ul\xi_k,\ol\xi_k$ and $m_{k,l}$}\label{subsec.xi}
\quad

To establish the existence of the solution of \eref{eqn.hqe},
by the Perron's method, the key point is
to find some appropriate subsolutions of the equation.
Since the Hessian quotient equation \eref{eqn.hqe}
is a highly fully nonlinear equation which including polynomials
of the eigenvalues of the the Hessian matrix $D^2u$,
$\sigma_k(\lambda)$ and $\sigma_l(\lambda)$,
of different order of homogeneities,
to solve it we need to strike a balance between them.
It will turn out to be clear that
the quantities $\Xi_k,\ul\xi_k,\ol\xi_k$ and $m_{k,l}$,
which we shall introduce below,
are very natural and perfectly fit for this purpose.
\begin{definition}
For any $0\leq k\leq n$ and any $a\in\R^n\setminus\{0\}$, let
\[\Xi_k:=\Xi_k(a,x)
:=\frac{\sum_{i=1}^n\sigma_{k-1;i}(a)a_i^2x_i^2}
{\sigma_k(a)\sum_{i=1}^{n}a_ix_i^2},
~\forall x\in\R^n\setminus\{0\},\]
and define
\[\ol{\xi}_k:=\ol{\xi}_k(a)
:=\sup_{x\in\R^n\setminus\{0\}}\Xi_k(a,x)\]
and
\[\ul{\xi}_k:=\ul{\xi}_k(a)
:=\inf_{x\in\R^n\setminus\{0\}}\Xi_k(a,x).\]
\end{definition}
\begin{definition}
For any $0\leq l<k\leq n$ and any $a\in\R^n\setminus\{0\}$, let
\begin{equation}\label{eqn.mkla}
\dps m_{k,l}:=m_{k,l}(a):=\frac{k-l}{\ol{\xi}_k(a)-\ul{\xi}_l(a)}.
\end{equation}
\end{definition}
We remark, for the reader's convenience, that
$\Xi_k$ originates from
the computation of $\sigma_k(D^2\Phi(x))$
where $\Phi(x)$ is a generalized radially symmetric function
(see \lref{lem.skm} and the proof of \lref{lem.Phi-subsol}),
that $\ul\xi_k$ and $\ol\xi_k$ result from
the comparison between $\sigma_k(\lambda)$ and $\sigma_l(\lambda)$
in the attempt to derive
an ordinary differential equation
from the original equation
(see the last part of the proof of \lref{lem.Phi-subsol}),
and that $m_{k,l}$ arises in the process of solving
this ordinary differential equation
(see \eref{eqn.mdlnr} in the proof of \lref{lem.psi}).
By $\Xi_k,\ul\xi_k$ and $\ol\xi_k$,
we get a good balance between
$\sigma_k(\lambda)$ and $\sigma_l(\lambda)$,
which can be measured by $m_{k,l}$.
Furthermore, we will find that
$m_{k,l}$ has also some special meaning related to
the decay and asymptotic behavior of the solution
(see \lref{lem.psi}\textsl{-(iii)},
\cref{cor.mub} and \tref{thm.hqe}).

\bigskip

It is easy to see that
\[\ol{\xi}_k(\varrho a)=\ol{\xi}_k(a),
~\ul{\xi}_k(\varrho a)=\ul{\xi}_k(a),
~\forall\varrho\neq 0,
~\forall a\in\R^n\setminus\{0\},
~\forall 0\leq k\leq n,\]
and
\[\ul\xi_k(C(1,1,...,1))=\frac{k}{n}
=\ol\xi_k(C(1,1,...,1)),
~\forall C>0,~\forall 0\leq k\leq n.\]
Furthermore, we have the following lemma.
\begin{lemma}\label{lem.xik}
Suppose $a=(a_1,a_2,...,a_n)$
with $0<a_1\leq a_2\leq...\leq a_n$.
Then
\begin{equation}\label{eqn.uxknox}
0<\frac{a_1\sigma_{k-1;1}(a)}{\sigma_k(a)}=\ul{\xi}_k(a)
\leq\frac{k}{n}\leq\ol{\xi}_k(a)
=\frac{a_n\sigma_{k-1;n}(a)}{\sigma_k(a)}\leq 1,
~\forall 1\leq k\leq n;
\end{equation}
\begin{equation}\label{eqn.olxiin}
0=\ol{\xi}_0(a)<\frac{1}{n}\leq\frac{a_n}{\sigma_1(a)}
=\ol{\xi}_1(a)\leq\ol{\xi}_2(a)
\leq...\leq\ol{\xi}_{n-1}(a)<\ol{\xi}_n(a)=1;
\end{equation}
and
\begin{equation}\label{eqn.ulxiin}
0=\ul{\xi}_0(a)<\frac{a_1}{\sigma_1(a)}
=\ul{\xi}_1(a)\leq\ul{\xi}_2(a)
\leq...\leq\ul{\xi}_{n-1}(a)<\ul{\xi}_n(a)=1.
\end{equation}
Moreover,
\begin{equation}\label{eqn.xkkn}
\ul\xi_k(a)=\frac{k}{n}=\ol\xi_k(a)
\end{equation}
for some $1\leq k\leq n-1$,
if and only if $a=C(1,1,...,1)$ for some $C>0$.
\end{lemma}

\begin{proof}
($1^\circ$)
By the definitions of $\sigma_k(a)$ and $\sigma_{k;i}(a)$,
we see that
\begin{equation}\label{eqn.sk}
\sigma_k(a)=\sigma_{k;i}(a)+a_i\sigma_{k-1;i}(a),
~\forall 1\leq i\leq n;
\end{equation}
and
\[\sum_{i=1}^{n}\sigma_{k;i}(a)
=\frac{nC_{n-1}^k}{C_n^k}\sigma_k(a)
=(n-k)\sigma_k(a).\]
Hence we obtain
\begin{equation}\label{eqn.ksk}
\sum_{i=1}^{n}a_i\sigma_{k-1;i}(a)=k\sigma_k(a).
\end{equation}

Now we show that
\begin{equation}\label{eqn.aiski}
a_1\sigma_{k-1;1}(a)
\leq a_2\sigma_{k-1;2}(a)\leq...
\leq a_n\sigma_{k-1;n}(a).
\end{equation}
In fact, for any $i\neq j$,
similar to \eref{eqn.sk}, we have
\[a_i\sigma_{k-1;i}(a)
=a_i\left(\sigma_{k-1;i,j}(a)+a_j\sigma_{k-2;i,j}(a)\right)\]
and
\[a_j\sigma_{k-1;j}(a)
=a_j\left(\sigma_{k-1;i,j}(a)+a_i\sigma_{k-2;i,j}(a)\right),\]
thus
\[a_i\sigma_{k-1;i}(a)-a_j\sigma_{k-1;j}(a)
=(a_i-a_j)\sigma_{k-1;i,j}(a).\]
Hence if $a_i\lessgtr a_j$, then
\begin{equation}\label{eqn.aslgeq}
a_i\sigma_{k-1;i}(a)\lessgtr a_j\sigma_{k-1;j}(a).
\end{equation}

By the definition of $\ol\xi_k$, we have
\begin{eqnarray*}
\dps \ol{\xi}_k(a)
&=&\sup_{x\neq 0}\frac{\sum_{i=1}^n\sigma_{k-1;i}(a)a_i^2x_i^2}
{\sigma_k(a)\sum_{i=1}^{n}a_ix_i^2}\\
\dps &\geq& \sup_{\substack{x_1=...=x_{n-1}=0,\\x_n\neq 0}}
\frac{\sum_{i=1}^n\sigma_{k-1;i}(a)a_i^2x_i^2}
{\sigma_k(a)\sum_{i=1}^{n}a_ix_i^2}\\
\dps &=&\sup_{x_n\neq 0}\frac{\sigma_{k-1;n}(a)a_n^2x_n^2}
{\sigma_k(a)a_n x_n^2}\\
\dps &=&\frac{a_n\sigma_{k-1;n}(a)}{\sigma_k(a)}
\end{eqnarray*}
and
\begin{eqnarray*}
\dps \ol{\xi}_k(a)
&=&\sup_{x\neq 0}\frac{\sum_{i=1}^n\sigma_{k-1;i}(a)a_i^2x_i^2}
{\sigma_k(a)\sum_{i=1}^{n}a_ix_i^2}\\
\dps &\leq& \sup_{x\neq 0}\frac{a_n\sigma_{k-1;n}(a)\sum_{i=1}^na_ix_i^2}
{\sigma_k(a)\sum_{i=1}^{n}a_ix_i^2}\qquad\text{by \eref{eqn.aiski}}\\
\dps &=&\frac{a_n\sigma_{k-1;n}(a)}{\sigma_k(a)}.
\end{eqnarray*}
Hence we obtain
\begin{equation}\label{eqn.olxik}
\ol\xi_k(a)=\frac{a_n\sigma_{k-1;n}(a)}{\sigma_k(a)}.
\end{equation}
Similarly
\begin{equation}\label{eqn.ulxik}
\ul\xi_k(a)=\frac{a_1\sigma_{k-1;1}(a)}{\sigma_k(a)}.
\end{equation}

From \eref{eqn.ksk}, we have
\[\sum_{i=1}^{n}\frac{a_i\sigma_{k-1;i}(a)}{\sigma_k(a)}=k.\]
Combining this with \eref{eqn.aiski},
\eref{eqn.olxik} and \eref{eqn.ulxik},
we deduce that
\[\ul\xi_k(a)\leq\frac{k}{n}\leq\ol\xi_k(a).\]
Thus the proof of \eref{eqn.uxknox} is complete,
and \eref{eqn.xkkn} is also clear in view of \eref{eqn.aslgeq}.

\medskip

($2^\circ$)
Since it follows from \eref{eqn.sk} that
\[a_i\sigma_{k-1;i}(a)<\sigma_k(a),
~\forall 1\leq i\leq n,
~\forall 1\leq k\leq n-1,\]
we obtain
\[\ul\xi_k(a)\leq\ol\xi_k(a)<1,~\forall 0\leq k\leq n-1.\]
On the other hand, we have $\ol{\xi}_n(a)=\ul{\xi}_n(a)=1$
which follows from
\[a_i\sigma_{n-1;i}(a)=\sigma_n(a),~\forall 1\leq i\leq n.\]

Combining \eref{eqn.olxik} and \eref{eqn.sk}, we discover that
\begin{eqnarray*}
\ol\xi_k(a)&=&\frac{a_n\sigma_{k-1;n}(a)}{\sigma_k(a)}
=\frac{a_n\sigma_{k-1;n}(a)}{\sigma_{k;n}(a)+a_n\sigma_{k-1;n}(a)}\\
&\leq&\frac{a_n\sigma_{k;n}(a)}{\sigma_{k+1;n}(a)+a_n\sigma_{k;n}(a)}
=\frac{a_n\sigma_{k;n}(a)}{\sigma_{k+1}(a)}
=\ol\xi_{k+1}(a),
\end{eqnarray*}
where we used the inequality
\[\frac{\sigma_{k-1;n}(a)}{\sigma_{k;n}(a)}
\leq\frac{\sigma_{k;n}(a)}{\sigma_{k+1;n}(a)}\]
which is a variation of the famous Newton inequality(see \cite{HLP34})
\[\sigma_{k-1}(\lambda)\sigma_{k+1}(\lambda)
\leq\left(\sigma_{k}(\lambda)\right)^2,
~\forall\lambda\in\R^n.\]
Thus the proof of \eref{eqn.olxiin},
and similarly of \eref{eqn.ulxiin}, is complete.
\end{proof}

\medskip

Since it follows from \eref{eqn.uxknox} that
\[\frac{k-l}{n}\leq\ol{\xi}_k(a)-\ul{\xi}_l(a)
<\ol{\xi}_k(a)\leq 1,\]
we obtain
\begin{corollary}\label{cor.m}
If $0\leq l<k\leq n$ and $a\in\Gamma^+$, then
\[1\leq k-l<m_{k,l}(a)\ol\xi_k(a)\leq m_{k,l}(a)\leq n.\]
\end{corollary}

As an application of \cref{cor.m} and \lref{lem.xik},
we now verify \pref{prop.wtakl}.
\begin{proof}[\textbf{Proof of \pref{prop.wtakl}}]
\textsl{(1)} and \textsl{(2)} are clear.
For \textsl{(3)}, we only need to note that
$c_\ast I\in\mathscr{A}_{k,l}$
and $m_{k,l}(c_\ast(1,1,...,1))=n>2$.
\end{proof}

\medskip

To help the reader to become familiar with these new quantities,
it is worth to give the following examples
which are also the applications of the above lemma.
\begin{example}
Note that, for $a=(a_1,a_2,a_3)\in\R^3$
with $0<a_1\leq a_2\leq a_3$, by \lref{lem.xik}, we have
\[\ol\xi_3(a)\equiv1\equiv\ul\xi_3(a),\]
\[\ol\xi_2(a)=\frac{a_3(a_1+a_2)}{a_1a_2+a_1a_3+a_2a_3},
\quad
\ul\xi_2(a)=\frac{a_1(a_2+a_3)}{a_1a_2+a_1a_3+a_2a_3},\]
\[\ol\xi_1(a)=\frac{a_3}{a_1+a_2+a_3},
\quad
\ul\xi_1(a)=\frac{a_1}{a_1+a_2+a_3},\]
and
\[\ol\xi_0(a)\equiv0\equiv\ul\xi_0(a).\]
Thus we can compute, for $a=(1,2,3)$, that
\[\ol\xi_2=\frac{9}{11},~\ul\xi_2=\frac{5}{11},
~\ol\xi_1=\frac{1}{2},~\ul\xi_1=\frac{1}{6},\]
\[m_{3,2}=\frac{11}{6}<2,
~m_{3,1}=\frac{12}{5}>2,
~m_{3,0}\equiv3>2,\]
\[m_{2,1}=\frac{66}{43}<2,
~m_{2,0}=\frac{22}{9}>2
~\text{and}~m_{1,0}=2,\]
and, for $a=(11,12,13)$, that
\[\ol\xi_2=\frac{299}{431},~\ul\xi_2=\frac{275}{431},
~\ol\xi_1=\frac{13}{36},~\ul\xi_1=\frac{11}{36},\]
\[m_{3,2}=\frac{431}{156}>2,
~m_{3,1}=\frac{72}{25}>2,
~m_{3,0}\equiv3>2,\]
\[m_{2,1}=\frac{15516}{6023}>2,
~m_{2,0}=\frac{862}{299}>2
~\text{and}~m_{1,0}=\frac{36}{13}>2.\]
\end{example}

\begin{remark}
\begin{enumerate}[(1)]
\item
By definition of $m_{k,l}$,
we can easily check that for any $1<k\leq n$,
$m_{k,k-1}(a)>2$
if and only if
$\ul\xi_{k-1}(a)\leq\ol\xi_{k}(a)\leq\ul\xi_{k-1}(a)+1/2$.
This will show us how $m_{k,l}$ plays a role
in the making of a balance
between different order of homogeneities
as we stated in the beginning of this subsection.

\item
\pref{prop.wtakl}-\textsl{(1)} states that
$\mathscr{\wt A}_{k,l}=\mathscr{A}_{k,l}$ provided $k-l\geq2$.
Note that this is the best case we can expect,
since in general
$\mathscr{\wt A}_{k,k-1}\subsetneqq\mathscr{A}_{k,k-1}$,
which is evident
by the fact stated in the first item of this remark
(and also by the above examples).
For example, in $\R^3$ we have
\[m_{3,2}(a)>2
\Leftrightarrow
\ol\xi_{3}(a)\leq\ul\xi_{2}(a)+1/2
\Leftrightarrow
a_1>\frac{a_2a_3}{a_2+a_3},\]
where the last inequality is not always true.\qed
\end{enumerate}
\end{remark}

\subsection{Some preliminary lemmas}\label{subsec.prelem}
\quad

In this subsection, we collect some preliminary lemmas
which will be mainly used in \sref{sec.pmaint}.

\medskip

We first give a lemma to compute $\sigma_k(\lambda(M))$
with $M$ of certain type.
~If $\Phi(x):=\phi(r)$ with $\phi\in C^2$, $r=\sqrt{x^TAx}$,
$A\in S(n)\cap\Gamma^+$ and $a=\lambda(A)$
(we may call $\Phi$ a \emph{generalized radially symmetric function}
with respect to $A$, according to \cite{BLL14}),
one can conclude that
\[\p_{ij}\Phi(x)
=\frac{\phi'(r)}{r}a_i\delta_{ij}+
\frac{\phi''(r)-\frac{\phi'(r)}{r}}{r^2}(a_ix_i)(a_jx_j),
~\forall 1\leq i,j\leq n,\]
provided $A$ is normalized to a diagonal matrix
(see the first part of \ssref{subsec.pmaint}
and the proof of \lref{lem.Phi-subsol} for details).
As far as we know, generally there is no explicit formula
for $\lambda(D^2\Phi(x))$ of this type,
but luckily we have a method
to calculate $\sigma_k\left(\lambda(D^2\Phi(x))\right)$
for each $1\leq k\leq n$,
which can be presented as follows.
\begin{lemma}\label{lem.skm}
If $M=\left(p_i\delta_{ij}+s q_iq_j\right)_{n\times n}$
with $p,q\in\R^n$ and $s\in\R$,
then
\[\sigma_k\left(\lambda(M)\right)
=\sigma_k(p)+s\sum_{i=1}^n\sigma_{k-1;i}(p)q_i^2,
~\forall 1\leq k\leq n.\]
\end{lemma}
\begin{proof}
See \cite{BLL14}.
\end{proof}

\medskip

To process information on the boundary
we need the following lemma.
\begin{lemma}\label{lem.Qxi}
Let $D$ be a bounded strictly convex domain
of $\R^n$, $n\geq 2$, $\p D\in C^2$,
$\varphi\in C^0(D)\cap C^2(\p{D})$
and let $A\in S(n)$, $\det{A}\neq0$.
Then there exists a constant $K>0$ depending only on $n$,
$\mbox{\emph{diam}}\,D$, the convexity of $D$,
$\norm{\varphi}_{C^2(\ol{D})}$,
the $C^2$ norm of $\p D$ and the upper bound of $A$,
such that for any $\xi\in\p D$,
there exists $\bar{x}(\xi)\in\R^n$ satisfying
\[\left|\bar{x}(\xi)\right|\leq K
\quad \mbox{and} \quad Q_\xi(x)<\varphi(x),
~\forall x\in \ol{D}\setminus\{\xi\},\]
where
\[Q_\xi(x):=\frac{1}{2}\left(x-\bar{x}(\xi)\right)^TA\left(x-\bar{x}(\xi)\right)
-\frac{1}{2}\left(\xi-\bar{x}(\xi)\right)^TA\left(\xi-\bar{x}(\xi)\right)
+\varphi(\xi),~\forall x\in\R^n.\]
\end{lemma}
\begin{proof}
See \cite{CL03} or \cite{BLL14}.
\end{proof}

\begin{remark}\label{rmk.Qxi}
It is easy to check that
$Q_\xi$ satisfy the following properties.
\begin{enumerate}[\quad(1)]
\item $Q_\xi\leq\varphi$ on $\ol{D}$
and $Q_\xi(\xi)=\varphi(\xi)$.

\item If $A\in\mathscr{A}_{k,l}$, then
\[\frac{\sigma_k(\lambda(D^2Q_\xi))}{\sigma_l(\lambda(D^2Q_\xi))}=1
\quad\mbox{in}~\R^n.\]

\item There exists $\bar{c}=\bar{c}(D,A,K)>0$ such that
\[Q_\xi(x)\leq\frac{1}{2}x^TAx+\bar{c},
\quad \forall x\in\p D,~\forall\xi\in\p D.\]
\end{enumerate}
\end{remark}

\medskip

Now we introduce the following well known lemmas
about the comparison principle and Perron's method
which will be applied to the Hessian quotient equations
but stated in a slightly more general setting.
These lemmas are adaptions of those appeared
in \cite{CNS85} \cite{Jen88} \cite{Ish89}
\cite{Urb90} and \cite{CIL92}.
For specific proof of them one may also consult
\cite{BLL14} and \cite{LB14}.
\begin{lemma}[Comparison principle]\label{lem.cp}
Assume $\Gamma^+\subset\Gamma\subset\R^n$ is
an open convex symmetric cone with its vertex at the origin,
and suppose $f\in C^1(\Gamma)$ and
$f_{\lambda_i}(\lambda)>0$, $\forall \lambda\in\Gamma$,
$\forall i=1,2,...,n$.
Let $\Omega\subset\R^n$ be a domain
and let $\ul{u},\ol{u}\in C^0(\ol\Omega)$ satisfying
\[f\left(\lambda\left(D^2\ul{u}\right)\right)\geq 1
\geq f\left(\lambda\left(D^2\ol{u}\right)\right)\]
in $\Omega$ in the viscosity sense.
Suppose $\ul{u}\leq \ol{u}$ on $\p\Omega$
(and additionally
\[\lim_{|x|\ra+\infty}\left(\ul{u}-\ol{u}\right)(x)=0\]
provided $\Omega$ is unbounded).
Then $\ul{u}\leq \ol{u}$ in $\Omega$.
\end{lemma}

\begin{lemma}[Perron's method]\label{lem.pm}
Assume that $\Gamma^+\subset\Gamma\subset\R^n$ is
an open convex symmetric cone with its vertex at the origin,
and suppose $f\in C^1(\Gamma)$ and
$f_{\lambda_i}(\lambda)>0$, $\forall \lambda\in\Gamma$,
$\forall i=1,2,...,n$.
Let $\Omega\subset\R^n$ be a domain, $\varphi\in C^0(\p\Omega)$
and let $\ul{u},\ol{u}\in C^0(\ol\Omega)$ satisfying
\[f\left(\lambda\left(D^2\ul{u}\right)\right)\geq 1
\geq f\left(\lambda\left(D^2\ol{u}\right)\right)\]
in $\Omega$ in the viscosity sense.
Suppose $\ul{u}\leq \ol{u}$ in $\Omega$,
$\ul{u}=\varphi$ on $\p\Omega$
(and additionally
\[\lim_{|x|\ra+\infty}\left(\ul{u}-\ol{u}\right)(x)=0\]
provided $\Omega$ is unbounded).
Then
\begin{eqnarray*}
u(x)&:=&\sup\Big\{v(x)\big{|}v\in C^0(\Omega),~
\ul{u}\leq v\leq \ol{u}~\mbox{in}~\Omega,~
f\left(\lambda\left(D^2v\right)\right)\geq 1~\mbox{in}~\Omega\\
&~&\qquad\mbox{in the viscosity sense},~
v=\varphi~\mbox{on}~\p\Omega\Big\}
\end{eqnarray*}
is the unique viscosity solution of the Dirichlet problem
\[\left\{
\begin{aligned}
f\left(\lambda\left(D^2u\right)\right)=1 \qquad &\mbox{in}& \Omega,\\
u=\varphi \qquad &\mbox{on}& \p\Omega.
\end{aligned}\right.
\]
\end{lemma}

\begin{remark}
In order to apply the above lemmas
to the Hessian quotient operator
\[f(\lambda):=\frac{\sigma_k(\lambda)}{\sigma_l(\lambda)}\]
in the cone $\Gamma:=\Gamma_k$,
we need to show that
\begin{equation}\label{eqn.pskl}
\p_{\lambda_i}\left(\frac{\sigma_k(\lambda)}{\sigma_l(\lambda)}\right)
>0,~\forall 1\leq i\leq n,
~\forall 0\leq l<k\leq n,
~\forall\lambda\in\Gamma_k,
\end{equation}
which indeed indicates that
the Hessian quotient equations \eref{eqn.hqe}
are elliptic equations with respect
to its $k$-convex solution $u$.

Indeed, for $l=0$, \eref{eqn.pskl} is clear
in light of \eref{eqn.pisjp}.
For $1\leq l<k\leq n$, since
\[\p_{\lambda_i}\sigma_k(\lambda)
=\frac{\sigma_k(\lambda)-\sigma_{k;i}(\lambda)}{\lambda_i}
=\sigma_{k-1;i}(\lambda)\]
according to \eref{eqn.sk}, we have
\[\p_{\lambda_i}\left(\frac{\sigma_k(\lambda)}
{\sigma_l(\lambda)}\right)
=\frac{\sigma_{k-1;i}(\lambda)\sigma_l(\lambda)
-\sigma_k(\lambda)\sigma_{l-1;i}(\lambda)}
{(\sigma_l(\lambda))^2}.\]
Thus to prove \eref{eqn.pskl}, it remains to verify
\[\sigma_{k-1;i}(\lambda)\sigma_l(\lambda)
\geq\sigma_k(\lambda)\sigma_{l-1;i}(\lambda).\]
In view of \eref{eqn.sk}, this is equivalent to
\[\sigma_{k-1;i}(\lambda)\sigma_{l;i}(\lambda)
\geq\sigma_{k;i}(\lambda)\sigma_{l-1;i}(\lambda),\]
which in turn is equivalent to
\[\frac{\sigma_{l;i}(\lambda)}{\sigma_{l-1;i}(\lambda)}
\geq\frac{\sigma_{k;i}(\lambda)}{\sigma_{k-1;i}(\lambda)},\]
since
$\sigma_{j;i}(\lambda)=\p_{\lambda_i}\sigma_{j+1}(\lambda)>0$,
$\forall 1\leq i\leq n$,
$\forall 0\leq j\leq k-1$,
$\forall\lambda\in\Gamma_k$,
according to \eref{eqn.pisjp}.
For the proof of the latter, we only need to note that
\[\frac{\sigma_{j;i}(\lambda)}{\sigma_{j-1;i}(\lambda)}
\geq\frac{\sigma_{j+1;i}(\lambda)}{\sigma_{j;i}(\lambda)},\]
which is the variation of the Newton inequality(see \cite{HLP34})
\[\sigma_{j-1}(\lambda)\sigma_{j+1}(\lambda)
\leq\left(\sigma_{j}(\lambda)\right)^2,
~\forall\lambda\in\R^n,\]
as we met in the proof of \lref{lem.xik}.\qed
\end{remark}

\section{Proof of the main theorem}\label{sec.pmaint}

\subsection{Construction of the subsolutions}\label{subsec.csubsol}
\quad

The purpose of this subsection is to prove
the following key lemma and then use it
to construct subsolutions of \eref{eqn.hqe}.
We remark that for the generalized radially symmetric subsolution
$\Phi(x)=\phi(r)$ that we intend to construct,
the solution $\psi(r)$ discussed in the the following lemma
actually is equivalent to $\phi'(r)/r$
(see the proof of the \lref{lem.Phi-subsol}).
\begin{lemma}\label{lem.psi}
Let $0\leq l<k\leq n$, $n\geq3$, $A\in\mathscr{\wt A}_{k,l}$,
$a:=(a_1,a_2,...,a_n):=\lambda(A)$,
$0<a_1\leq a_2\leq...\leq a_n$ and $\beta\geq 1$.
Then the problem
\begin{equation}\label{eqn.psi}
\left\{
\begin{aligned}
\psi(r)^k+\ol{\xi}_k(a)r\psi(r)^{k-1}\psi'(r)\quad&\\
-\psi(r)^l-\ul{\xi}_l(a)r\psi(r)^{l-1}\psi'(r)&=0,~r>1,\\[0.1cm]
\psi(1)&=\beta,
\end{aligned}
\right.
\end{equation}
has a unique smooth solution $\psi(r)=\psi(r,\beta)$ on $[1,+\infty)$,
which satisfies
\begin{enumerate}[\quad(i)]
\item[(i)] $1\leq\psi(r,\beta)\leq\beta$, $\p_r\psi(r,\beta)\leq0$,
$\forall r\geq1$, $\forall\beta\geq1$.
More specifically,
$\psi(r,1)\equiv 1$, $\psi(1,\beta)\equiv\beta$;
and $1<\psi(r,\beta)<\beta$, $\forall r>1$, $\forall\beta>1$.

\item[(ii)] $\psi(r,\beta)$ is continuous
and strictly increasing with respect to $\beta$
and \[\lim_{\beta\ra+\infty}\psi(r,\beta)=+\infty,~\forall r\geq 1.\]

\item[(iii)] $\psi(r,\beta)=1+O(r^{-m})~(r\ra+\infty)$,
where $m=m_{k,l}(a)\in(2,n]$
and the $O(\cdot)$ depends only on
$k$, $l$, $\lambda(A)$ and $\beta$.
\end{enumerate}
\end{lemma}

\begin{proof}
For brevity, we will often write
$\psi(r)$ or $\psi(r,\beta)$ (respectively, $\ul\xi(a),\ol\xi(a)$)
simply as $\psi$ (respectively, $\ul\xi,\ol\xi$),
when there is no confusion.
The proof of this lemma now will be divided into three steps.

\medskip

\emph{Step 1.}\quad
We deduce from \eref{eqn.psi} that
\begin{equation}\label{eqn.psi-kl}
\psi^k-\psi^l=-\frac{r}{dr}\left(\ol{\xi}_k\psi^{k-1}
-\ul{\xi}_l\psi^{l-1}\right)d\psi
\end{equation}
and
\begin{equation}\label{eqn.psid}
\frac{d\psi}{dr}=-\frac{1}{r}\cdot\frac{\psi^k-\psi^l}
{\ol{\xi}_k\psi^{k-1}-\ul{\xi}_l\psi^{l-1}}
=-\frac{1}{r}\cdot\frac{\psi}{\ol{\xi}_k}\cdot
\frac{\psi^{k-l}-1}{\psi^{k-l}-\frac{\ul{\xi}_l}{\ol{\xi}_k}}
=:\frac{g(\psi)}{r},
\end{equation}
where we set
\[g(\nu):=-\frac{\nu}{\ol{\xi}_k}\cdot
\frac{\nu^{k-l}-1}{\nu^{k-l}-\frac{\ul{\xi}_l}{\ol{\xi}_k}}.\]
Hence the problem \eref{eqn.psi}
is equivalent to the following problem
\begin{equation}\label{eqn.psi-g}
\left\{
\begin{aligned}
\psi'(r)&=\frac{g(\psi(r))}{r},~r>1,\\
\psi(1)&=\beta.
\end{aligned}
\right.
\end{equation}

If $\beta=1$, then $\psi(r)\equiv 1$ is a solution
of the problem \eref{eqn.psi-g} since $g(1)=0$.
Thus, by the uniqueness theorem
for the solution of the ordinary differential equation,
we know that $\psi(r,1)\equiv 1$ is the unique solution
satisfies the problem \eref{eqn.psi-g}.

Now if $\beta>1$, since
\[h(r,\nu):=\frac{g(\nu)}{r}
\in C^{\infty}((1,+\infty)\times(\nu_0,+\infty)),\]
where \[\frac{\ul{\xi}_l}{\ol{\xi}_k}<\nu_0<1\]
(note that $\nu_0$ exists, since we have
$\ul{\xi}_l\leq l/n<k/n\leq\ol{\xi}_k$ by \lref{lem.xik}),
by the existence theorem
(the Picard-Lindel\"{o}f theorem)
and the theorem of the maximal interval of existence
for the solution
of the initial value problem
of the ordinary differential equation,
we know that the problem \eref{eqn.psi-g}
has a unique smooth solution $\psi(r)=\psi(r,\beta)$
locally around the initial point
and can be extended to a maximal interval
$[1,\zeta)$ in which $\zeta$
can only be one of the following cases:
\begin{enumerate}[\qquad($1^\circ$)]
\item $\zeta=+\infty$;

\item $\zeta<+\infty$, $\psi(r)$ is unbounded on $[1,\zeta)$;

\item $\zeta<+\infty$, $\psi(r)$ converges to some point on $\{\nu=\nu_0\}$
as $r\ra\zeta-$.
\end{enumerate}
Since \[\frac{g(\psi(r))}{r}<0,~\forall \psi(r)>1,\]
we see that $\psi(r)=\psi(r,\beta)$
is strictly decreasing with respect to $r$
which exclude the case ($2^\circ$) above.
We claim now that
the case ($3^\circ$) can also be excluded.
Otherwise, the solution curve
must intersect with $\{\nu=1\}$ at some point $(r_0,\psi(r_0))$ on it
and then tends to $\{\nu=\nu_0\}$ after crossing it.
But $\psi(r)\equiv 1$ is also a solution through $(r_0,\psi(r_0))$
which contradicts the uniqueness theorem
for the solution
of the initial value problem
of the ordinary differential equation.
Thus we complete the proof of the existence and uniqueness
of the solution $\psi(r)=\psi(r,\beta)$
of the problem \eref{eqn.psi} on $[1,+\infty)$.

Due to the same reason, i.e.,
$\psi(r,\beta)$ is strictly decreasing with respect to $r$
and the solution curve can not cross $\{\nu=1\}$ provided $\beta>1$,
assertion $\emph{(i)}$ of the lemma is also clear now,
that is, $1<\psi(r,\beta)<\beta$, $\forall r>1$, $\forall\beta>1$.

\medskip

\emph{Step 2.}\quad
By the theorem of the differentiability of the solution
with respect to the initial value,
we can differentiate $\psi(r,\beta)$
with respect to $\beta$ as blew:
\[\left\{
\begin{aligned}
&\frac{\p\psi(r,\beta)}{\p r}=\frac{g(\psi(r,\beta))}{r},&\\
&\psi(1,\beta)=\beta;&
\end{aligned}\right.\]
\[\Ra\left\{
\begin{aligned}
&\frac{\p^2 \psi(r,\beta)}{\p \beta\p r}
=\frac{g'(\psi(r,\beta))}{r}\cdot\frac{\p\psi(r,\beta)}{\p\beta},&\\
&\frac{\p\psi(1,\beta)}{\p\beta}=1.&
\end{aligned}\right.\]
Let
\[v(r):=\frac{\p\psi(r,\beta)}{\p\beta}.\]
We have
\[\left\{
\begin{aligned}
&\frac{dv}{dr}
=\frac{g'(\psi(r,\beta))}{r}\cdot v,&\\
&v(1)=1.&
\end{aligned}\right.\]
Therefore we can deduce that
\[\frac{dv}{v}
=\frac{g'(\psi(r,\beta))}{r}dr,\]
and hence
\[\frac{\p\psi(r,\beta)}{\p\beta}=v(r)
=\exp\int_1^r{\frac{g'(\psi(\tau,\beta))}{\tau}d\tau}.\]

Since
\[g'(\nu)=-\frac{\nu^{k-l}-1}{\ol{\xi}_k\nu^{k-l}-\ul{\xi}_l}
-\frac{\nu}{\ol{\xi}_k}\cdot
\frac{-\left(\frac{\ul{\xi}_l}{\ol{\xi}_k}-1\right)(k-l)\nu^{k-l-1}}
{\left(\nu^{k-l}-\frac{\ul{\xi}_l}{\ol{\xi}_k}\right)^2}\]
and
\begin{equation}\label{eqn.onepsibeta}
0<\frac{\ul{\xi}_l}{\ol{\xi}_k}
<1\leq\psi(r,\beta)\leq\beta=\psi(1),~\forall r\geq 1,
\end{equation}
we have
\[g'(\psi(r,\beta))\leq
-\frac{(k-l)\left(1-\frac{\ul{\xi}_l}{\ol{\xi}_k}\right)}
{\ol{\xi}_k
\left(\beta^{k-l}-\frac{\ul{\xi}_l}{\ol{\xi}_k}\right)^2}
=-C\left(k,l,\lambda(A),\beta\right)<0,\]
and hence
\[0<\frac{\p\psi(r,\beta)}{\p\beta}
\leq r^{-C}\leq 1,~\forall r\geq 1.\]
Thus $\psi(r,\beta)$ is strictly increasing with respect to $\beta$.

\medskip

\emph{Step 3.}\quad
By \eref{eqn.psi-kl}, we have
\begin{eqnarray*}
-d\ln r=-\frac{dr}{r}&=&\frac{\ol{\xi}_k\psi^{k-1}
-\ul{\xi}_l\psi^{l-1}}{\psi^k-\psi^l}d\psi
=\frac{\ol{\xi}_k}{\psi}\cdot
\frac{\psi^{k-l}-\frac{\ul{\xi}_l}{\ol{\xi}_k}}{\psi^{k-l}-1}d\psi\\
&=&\frac{\ol{\xi}_k}{\psi}\left(1
+\frac{1-\frac{\ul{\xi}_l}{\ol{\xi}_k}}{\psi^{k-l}-1}\right)d\psi
=\left(\frac{\ol{\xi}_k}{\psi}
+\frac{\ol{\xi}_k-\ul{\xi}_l}{\psi(\psi^{k-l}-1)}\right)d\psi\\
&=&\ol{\xi}_kd\ln\psi-\frac{\ol{\xi}_k-\ul{\xi}_l}{k-l}
d\ln\frac{\psi^{k-l}}{\psi^{k-l}-1}\\
&=&d\ln\left(\psi^{\ol{\xi}_k}\left(1-\psi^{-k+l}\right)
^{\frac{\ol{\xi}_k-\ul{\xi}_l}{k-l}}\right).
\end{eqnarray*}
Hence
\begin{equation}\label{eqn.mdlnr}
-md\ln r=d\ln\left(\psi(r)^{m\ol{\xi}_k}
\left(1-\psi(r)^{-k+l}\right)\right),
\end{equation}
where
\[m:=m_{k,l}(a):=\frac{k-l}{\ol{\xi}_k-\ul{\xi}_l},\]
which has been already defined
in \eref{eqn.mkla} in \ssref{subsec.xi}.
Note that, by the assumptions on $A$ and \cref{cor.m},
we have $2<m\leq n$ and $m\ol{\xi}_k>k-l$.
Integrating \eref{eqn.mdlnr} from $1$ to $r$
and recalling $\psi(1)=\beta\geq1$,
we get
\[\ln\left(\psi(r)^{m\ol{\xi}_k}
\left(1-\psi(r)^{-k+l}\right)\right)
=\ln\left(\beta^{m\ol{\xi}_k}
\left(1-\beta^{-k+l}\right)\right)+\ln r^{-m},\]
and hence
\[\psi(r)^{m\ol{\xi}_k}
\left(1-\psi(r)^{-k+l}\right)
=\beta^{m\ol{\xi}_k}
\left(1-\beta^{-k+l}\right)r^{-m}:=B(\beta)r^{-m},\]
where we set
\[B(\beta):=\beta^{m\ol{\xi}_k}\left(1-\beta^{-k+l}\right)
=\beta^{m\ol{\xi}_k-k+l}\left(\beta^{k-l}-1\right).\]
Since
\begin{eqnarray*}
&~&\psi(r)^{m\ol{\xi}_k}
\left(1-\psi(r)^{-k+l}\right)
=\psi(r)^{m\ol{\xi}_k-k+l}
\left(\psi(r)^{k-l}-1\right)\\
&=&\psi(r)^{m\ol{\xi}_k-k+l}\left(\psi(r)-1\right)
\left(\psi(r)^{k-l-1}+\psi(r)^{k-l-2}+...+\psi(r)+1\right),
\end{eqnarray*}
we thus conclude that
\begin{equation}\label{eqn.psimoar}
\frac{\psi(r)-1}{r^{-m}}
=\left(\psi(r)^{m\ol{\xi}_k-k+l}
\left(\psi(r)^{k-l-1}+\psi(r)^{k-l-2}+
...+\psi(r)+1\right)\right)^{-1}B(\beta).
\end{equation}
Note that $m\ol{\xi}_k-k+l>0$ and
\[\beta-1=\left(\beta^{m\ol{\xi}_k-k+l}
\left(\beta^{k-l-1}+\beta^{k-l-2}+
...+b+1\right)\right)^{-1}B(\beta).\]
Recalling \eref{eqn.onepsibeta}, we obtain
\begin{equation}\label{eqn.psi-b}
\beta-1\leq\frac{\psi(r,\beta)-1}{r^{-m}}
\leq\frac{B(\beta)}{k-l},
~\forall r\geq 1.
\end{equation}
Thus we have
\[\lim_{\beta\ra+\infty}\psi(r,\beta)=+\infty,
~\forall r\geq 1,\]
and
\[\psi(r,\beta)\ra 1~(r\ra+\infty),
~\forall \beta\geq 1.\]
Substituting the latter to \eref{eqn.psimoar}, we get
\[\frac{\psi(r,\beta)-1}{r^{-m}}
\ra\frac{B(\beta)}{k-l}~(r\ra+\infty),
~\forall \beta\geq 1.\]
Therefore
\[\psi(r,\beta)=1+\frac{B(\beta)}{k-l}r^{-m}+o(r^{-m})
=1+O(r^{-m})~(r\ra+\infty),\]
where $o(\cdot)$ and $O(\cdot)$
depend only on $k$, $l$, $\lambda(A)$ and $\beta$.
This completes the proof of the lemma.
\end{proof}


\begin{remark}
For $l=0$, i.e., the Hessian equation $\sigma_k(\lambda)=1$,
we have an easy proof. Consider the problem
\begin{equation}\label{eqn.psi-he}
\left\{
\begin{aligned}
\psi(r)^k+\ol{\xi}_k(a)r\psi(r)^{k-1}\psi'(r)&=1,~r>1,\\[0.1cm]
\psi(1)&=\beta.
\end{aligned}
\right.
\end{equation}
Set $m:=m_{k,0}(a)=k/\ol\xi_k$.
We have
\[\psi^k-1=-r\ol\xi_k\psi^{k-1}\frac{d\psi}{dr}
=-\frac{1}{m}\cdot r\cdot\frac{d\left(\psi^k-1\right)}{dr},\]
\[\frac{d\left(\psi^k-1\right)}{\psi^k-1}
=-m\frac{dr}{r}\]
and
\[d\ln\left(\psi(r)^k-1\right)=-md\ln r=d\ln r^{-m}.\]
Integrating it from $1$ to $r$
and recalling $\psi(1)=\beta\geq1$,
we get
\[\psi(r)^k-1=\left(\psi(1)^k-1\right)r^{-m}
=(\beta^k-1)r^{-m}\]
and
\begin{eqnarray}
\nonumber\psi(r)&=&\left(1+(\beta^k-1)r^{-m}\right)^{\frac{1}{k}}\\
&=&\left(1+(\beta^k-1)r^{-\frac{k}{\ol\xi_k}}\right)^{\frac{1}{k}}
\label{eqn.compare}\\
\nonumber&=&1+\frac{\beta^k-1}{k}r^{-m}+o(r^{-m})=1+O(r^{-m})~(r\ra+\infty).
\end{eqnarray}
It is obvious that the $\psi(r)$ that
we here solved from \eref{eqn.psi-he} for $l=0$
satisfies all the conclusions of \lref{lem.psi}.
Moreover, comparing \eref{eqn.compare} with
the corresponding ones in \cite{BLL14} and in \cite{CL03},
we observe that our method actually provides a systematic way
for construction of the subsolutions,
which gives results containing
the previous ones as special cases.
\qed
\end{remark}

\medskip

Set
\[\mu_R(\beta):=\int_R^{+\infty}\tau\big(\psi(\tau,\beta)-1\big)d\tau,
\quad\forall R\geq 1,~\forall\beta\geq 1.\]
Note that the integral on the right hand side
is convergent in view of \lref{lem.psi}-\textsl{(iii)}.
Moreover, as an application of \lref{lem.psi}, we have the following.
\begin{corollary}\label{cor.mub}
$\mu_R(\beta)$ is nonnegative, continuous and
strictly increasing with respect to $\beta$.
Furthermore,
\[\mu_R(\beta)\geq\int_R^{+\infty}(\beta-1)r^{-m+1}d\tau
\ra+\infty~(\beta\ra+\infty),~\forall R\geq 1;\]
and
\[\mu_R(\beta)=O(R^{-m+2})~(R\ra+\infty),~\forall\beta\geq 1.\]
\end{corollary}

\begin{proof}
By \lref{lem.psi}-\textsl{(ii),(iii)}
and the above property \eref{eqn.psi-b} of $\psi(r,\beta)$.
\end{proof}

\medskip

For any $\alpha,\beta,\gamma\in\R$, $\beta,\gamma\geq1$
and for any diagonal matrix $A\in\mathscr{\wt A}_{k,l}$, let
\[\phi(r):=\phi_{\alpha,\beta,\gamma}(r)
:=\alpha+\int_\gamma^r\tau\psi(\tau,\beta)d\tau,
~\forall r\geq\gamma,\]
and
\[\Phi(x):=\Phi_{\alpha,\beta,\gamma,A}(x):=\phi(r)
:=\phi_{\alpha,\beta,\gamma}(r_A(x)),
~\forall x\in\R^n\setminus{E_\gamma},\]
where $r=r_A(x)=\sqrt{x^TAx}$.
Then we have
\begin{eqnarray}
\phi_{\alpha,\beta,\gamma}(r)
&=&\nonumber\int_\gamma^r\tau\big(\psi(\tau,\beta)-1\big)d\tau
+\frac{1}{2}r^2-\frac{1}{2}\gamma^2+\alpha\\
&=&\frac{1}{2}r^2+\left(\mu_\gamma(\beta)+\alpha
-\frac{1}{2}\gamma^2\right)-\mu_r(\beta)\label{eqn.phi-mu}\\
&=&\frac{1}{2}r^2+\left(\mu_\gamma(\beta)+\alpha
-\frac{1}{2}\gamma^2\right)+O(r^{-m+2})
~(r\ra+\infty),\quad\label{eqn.phi-O}
\end{eqnarray}
according to \cref{cor.mub},
and now we can assert that
\begin{lemma}\label{lem.Phi-subsol}
$\Phi$ is a smooth $k$-convex subsolution of \eref{eqn.hqe}
in $\R^n\setminus\ol{E_\gamma}$, that is,
\[\sigma_j\left(\lambda\left(D^2{\Phi(x)}\right)\right)\geq0,
~\forall 1\leq j\leq k,~\forall x\in\R^n\setminus\ol{E_\gamma},\]
and
\[\frac{\sigma_k\left(\lambda\left(D^2{\Phi(x)}\right)\right)}
{\sigma_l\left(\lambda\left(D^2{\Phi(x)}\right)\right)}\geq 1,
~\forall x\in\R^n\setminus\ol{E_\gamma}.\]
\end{lemma}

\begin{proof}
By definition we have $\phi'(r)=r\psi(r)$
and $\phi''(r)=\psi(r)+r\psi'(r)$.
Since
\[r^2=x^TAx=\sum_{i=1}^{n}a_ix_i^2,\]
we deduce that
\[2r\p_{x_i}r=\p_{x_i}\left(r^2\right)=2a_ix_i
\quad\text{and}\quad
\p_{x_i}r=\frac{a_ix_i}{r}.\]
Consequently
\[\p_{x_i}\Phi(x)=\phi'(r)\p_{x_i}r
=\frac{\phi'(r)}{r}a_ix_i,\]
\begin{eqnarray*}
\p_{x_ix_j}\Phi(x)
&=&\frac{\phi'(r)}{r}a_i\delta_{ij}+
\frac{\phi''(r)-\frac{\phi'(r)}{r}}{r^2}(a_ix_i)(a_jx_j)\\
&=&\psi(r)a_i\delta_{ij}+\frac{\psi'(r)}{r}(a_ix_i)(a_jx_j),
\end{eqnarray*}
and therefore
\[D^2\Phi=\left(\psi(r)a_i\delta_{ij}
+\frac{\psi'(r)}{r}(a_ix_i)(a_jx_j)\right)_{n\times n}.\]
So we can conclude from \lref{lem.skm} that
\begin{eqnarray*}
\sigma_j\left(\lambda\left(D^2\Phi\right)\right)
&=&\sigma_j(a)\psi(r)^j+\frac{\psi'(r)}{r}\psi(r)^{j-1}
\sum_{i=1}^n\sigma_{j-1;i}(a)a_i^2x_i^2\\
&=&\sigma_j(a)\psi^j+\Xi_j(a,x)\sigma_j(a)r\psi^{j-1}\psi'\\
&\geq&\sigma_j(a)\psi^j+\ol{\xi}_j(a)\sigma_j(a)r\psi^{j-1}\psi'\\
&=&\sigma_j(a)\psi^{j-1}\left(\psi+\ol{\xi}_j(a)r\psi'\right),
~\forall 1\leq j\leq n,
\end{eqnarray*}
where we have used the facts that
$\psi(r)\geq1>0$ and $\psi'(r)\leq0$ for all $r\geq1$,
according to \lref{lem.psi}-\textsl{(i)}.

For any fixed $1\leq j\leq k$,
in view of \lref{lem.xik} and \lref{lem.psi}-\textsl{(i)},
we have
\[0\leq\frac{\psi^{k-l}-1}{\psi^{k-l}-\frac{\ul{\xi}_l(a)}{\ol{\xi}_k(a)}}
<1\leq\frac{\ol{\xi}_k(a)}{\ol{\xi}_j(a)}.\]
Hence it follows from \eref{eqn.psid} that
\[\psi'=-\frac{1}{r}\cdot\frac{\psi}{\ol{\xi}_k(a)}\cdot
\frac{\psi^{k-l}-1}{\psi^{k-l}-\frac{\ul{\xi}_l(a)}{\ol{\xi}_k(a)}}
>-\frac{1}{r}\cdot\frac{\psi}{\ol{\xi}_j(a)},\]
which yields $\psi+\ol{\xi}_j(a)r\psi'>0$.

Since $A\in\mathscr{\wt A}_{k,l}$
implies $a\in\Gamma^+$,
that is, $\sigma_i(a)>0$ for all $1\leq i\leq n$,
we thus conclude that
\[\sigma_j\left(\lambda\left(D^2\Phi\right)\right)>0,
~\forall 1\leq j\leq k.\]
In particular, we have
\[\sigma_k\left(\lambda\left(D^2\Phi\right)\right)>0
\quad\text{and}\quad
\sigma_l\left(\lambda\left(D^2\Phi\right)\right)>0.\]

On the other hand,
\begin{eqnarray*}
&~&\sigma_k\left(\lambda\left(D^2\Phi\right)\right)
-\sigma_l\left(\lambda\left(D^2\Phi\right)\right)\\
&=&\sigma_k(a)\psi^k+\Xi_k(a,x)\sigma_k(a)r\psi^{k-1}\psi'
-\sigma_l(a)\psi^l-\Xi_l(a,x)\sigma_l(a)r\psi^{l-1}\psi'\\
&\geq&\sigma_k(a)\psi^k+\ol{\xi}_k(a)\sigma_k(a)r\psi^{k-1}\psi'
-\sigma_l(a)\psi^l-\ul{\xi}_l(a)\sigma_l(a)r\psi^{l-1}\psi'\\
&=&\sigma_k(a)\left(\psi^k+\ol{\xi}_k(a)r\psi^{k-1}\psi'
-\psi^l-\ul{\xi}_l(a)r\psi^{l-1}\psi'\right)\\
&=&0.
\end{eqnarray*}
Therefore
\[\frac{\sigma_k\left(\lambda\left(D^2{\Phi}\right)\right)}
{\sigma_l\left(\lambda\left(D^2{\Phi}\right)\right)}\geq 1.\]
This completes the proof of \lref{lem.Phi-subsol}.
\end{proof}

\subsection{Proof of \tref{thm.hqe}}\label{subsec.pmaint}
\quad

We first introduce the following lemma
which is a special and simple case of \tref{thm.hqe}
with the additional condition that
the matrix $A$ is diagonal and the vector $b$ vanishes.
\begin{lemma}\label{lem.hqe}
Let $D$ be a bounded strictly convex domain
in $\R^n$, $n\geq 3$, $\p D\in C^2$
and let $\varphi\in C^2(\p D)$.
Then for any given \textsl{diagonal matrix}
$A\in\mathscr{\wt A}_{k,l}$
with $0\leq l<k\leq n$,
there exists a constant $\tilde{c}$
depending only on $n,D,k,l,A$
and $\norm{\varphi}_{C^2(\p D)}$,
such that for every $c\geq\tilde{c}$,
there exists a unique viscosity solution
$u\in C^0(\R^n\setminus D)$ of
\begin{equation}\label{eqn.hqe-ac}
\left\{
\begin{aligned}
\dps \frac{\sigma_k(\lambda(D^2u))}{\sigma_l(\lambda(D^2u))}&=1
\quad\text{in}~\R^n\setminus\ol{D},\\
\dps u&=\varphi\quad\text{on}~\p D,\\
\dps \limsup_{|x|\ra+\infty}|x|^{m-2}
&\left|u(x)-\left(\frac{1}{2}x^{T}Ax+c\right)\right|<\infty,
\end{aligned}
\right.
\end{equation}
where $m=m_{k,l}(\lambda(A))\in(2,n]$.
\end{lemma}

To prove \tref{thm.hqe}, it suffices to prove \lref{lem.hqe}.
Indeed, suppose that $D,\varphi,A$ and $b$
satisfy the hypothesis of \tref{thm.hqe}.
Consider the decomposition $A=Q^TNQ$,
where $Q$ is an orthogonal matrix and $N$ is a diagonal matrix
which satisfies $\lambda(N)=\lambda(A)$.
Let
\[\tilde{x}:=Qx,\quad\wt{D}:=\left\{Qx|x\in D\right\}\]
and
\[\tilde{\varphi}(\tilde{x}):=\varphi(x)-b^Tx
=\varphi(Q^T\tilde{x})-b^TQ^T\tilde{x}.\]
By \lref{lem.hqe}, we conclude that
there exists a constant $\tilde{c}$
depending only on $n,\wt{D},k,l,N$
and $\norm{\tilde{\varphi}}_{C^2(\p\wt{D})}$,
such that for every $c\geq\tilde{c}$,
there exists a unique viscosity solution
$\tilde{u}\in C^0(\R^n\setminus\wt{D})$ of
\begin{equation}\label{eqn.hqe-nc}
\left\{
\begin{aligned}
\dps \frac{\sigma_k(\lambda(D^2\tilde{u}))}
{\sigma_l(\lambda(D^2\tilde{u}))}&=1
\quad\text{in}~\R^n\setminus\ol{\wt{D}},\\
\dps \tilde{u}&=\tilde{\varphi}
\quad\text{on}~\p\wt{D},\\
\dps \limsup_{|\tilde{x}|\ra+\infty}|\tilde{x}|^{m-2}
&\left|\tilde{u}(\tilde{x})-\left(\frac{1}{2}\tilde{x}^{T}N\tilde{x}
+c\right)\right|<\infty,
\end{aligned}
\right.
\end{equation}
where $m=m_{k,l}(\lambda(N))=m_{k,l}(\lambda(A))\in(2,n]$.
Let
\[u(x):=\tilde{u}(\tilde{x})+b^Tx=\tilde{u}(Qx)+b^Tx
=\tilde{u}(\tilde{x})+b^TQ^T\tilde{x}.\]
We claim that
$u$ is the solution of \eref{eqn.hqe-abc} in \tref{thm.hqe}.
To show this, we only need to note that
\[D^2u(x)=Q^TD^2\tilde{u}(\tilde{x})Q,\quad
\lambda\left(D^2u(x)\right)=\lambda\left(D^2\tilde{u}(\tilde{x})\right);\]
\[u=\varphi\quad\text{on}~\p D;\]
and
\begin{eqnarray*}
&&\dps |\tilde{x}|^{m-2}\left|\tilde{u}(\tilde{x})
-\left(\frac{1}{2}\tilde{x}^{T}N\tilde{x}+c\right)\right|\\
&=&\dps \left(x^TQ^TQx\right)^{(m-2)/2}\left|u(x)-b^Tx
-\left(\frac{1}{2}x^TQ^TNQx+c\right)\right|\\
&=&\dps |x|^{m-2}
\left|u(x)-\left(\frac{1}{2}x^{T}Ax+b^Tx+c\right)\right|.
\end{eqnarray*}
Thus we have proved that
\tref{thm.hqe} can be established by \lref{lem.hqe}.
\begin{remark}
\begin{enumerate}[(1)]
\item
We may see from the above demonstration that
the lower bound $\tilde{c}$ of $c$ in \tref{thm.hqe}
can not be discarded generally.
Indeed, for the radial solutions of
the Hessian equation $\sigma_k(\lambda(D^2u))=1$
in $\R^n\setminus\ol{B_1}$,
\cite[Theorem 2]{WB13} states that
there is no solution when $c$ is too small.

\item
Unlike the Poisson equation
and the Monge-Amp\`{e}re equation,
generally, for the Hessian quotient equation,
the matrix $A$ in \tref{thm.hqe}
can only be normalized to a diagonal matrix,
and can not be normalized to $I$
multiplied by some constant.
This is the reason why we study
the generalized radially symmetric solutions,
rather than the radial solutions,
of the original equation \eref{eqn.hqe}.
See also \cite{BLL14}.\qed
\end{enumerate}
\end{remark}

\bigskip

Now we use the Perron's method to prove \lref{lem.hqe}.
\begin{proof}[\textbf{Proof of \lref{lem.hqe}}]
We may assume without loss of generality that
$E_1\subset\subset D\subset\subset E_{\bar{r}}
\subset\subset E_{\hat{r}}$
and $a:=(a_1,a_2,...,a_n):=\lambda(A)$ with
$0<a_1\leq a_2\leq...\leq a_n$.
The proof now will be divided into three steps.

\medskip

\emph{Step 1.}\quad
Let
\[\eta:=\inf_{\substack{x\in \ol{E_{\bar{r}}}\setminus D\\
\xi\in\p D}}Q_\xi(x), \quad
Q(x):=\sup_{\xi\in\p D}Q_\xi(x)\]
and
\[\Phi_\beta(x):=\eta+\int_{\bar{r}}^{r_A(x)}\tau\psi(\tau,\beta)d\tau,
\quad\forall r_A(x)\geq 1,~\forall \beta\geq 1,\]
where $Q_{\xi}(x)$ and $\psi(r,\beta)$ are given
by \lref{lem.Qxi} and \lref{lem.psi}, respectively.
Then we have
\begin{enumerate}[\quad(1)]
\item Since $Q$ is the supremum of
a collection of smooth solutions $\{Q_\xi\}$ of \eref{eqn.hqe},
it is a continuous subsolution of \eref{eqn.hqe}, i.e.,
\[\frac{\sigma_k(\lambda(D^2Q))}{\sigma_l(\lambda(D^2{Q}))}\geq 1\]
in $\R^n\setminus \ol{D}$ in the viscosity sense
(see \cite[Proposition 2.2]{Ish89}).

\item $Q=\varphi$ on $\p D$.
To prove this we only need to show that
for any $\xi\in\p D$, $Q(\xi)=\varphi(\xi)$.
This is obvious since $Q_\xi\leq\varphi$ on $\ol{D}$
and $Q_\xi(\xi)=\varphi(\xi)$,
according to \rref{rmk.Qxi}-\textsl{(1)}.

\item By \lref{lem.Phi-subsol},
$\Phi_\beta$ is a smooth subsolution of \eref{eqn.hqe}
in $\R^n\setminus\ol{D}$.

\item $\Phi_\beta\leq\varphi$ on $\p D$
and $\Phi_\beta\leq Q$ on $\ol{E_{\bar{r}}}\setminus D$.
To show them we first note that
$\Phi_\beta(x)$ is strictly increasing
with respect to $r_A(x)$ since
$\psi(r,\beta)\geq 1>0$ by \lref{lem.psi}-\textsl{(i)}.
Invoking $\Phi_\beta=\eta$ on $\p E_{\bar{r}}$
and $\eta\leq Q$ on $\ol{E_{\bar{r}}}\setminus D$
by their definitions, we have $\Phi_\beta\leq\eta\leq Q$
on $\ol{E_{\bar{r}}}\setminus D$.
On the other hand,
according to \rref{rmk.Qxi}-\textsl{(1)}, we have
$Q_\xi\leq\varphi$ on $\ol{D}$ which implies that
$\eta\leq\varphi$ on $\ol{D}$.
Combining these two aspects we deduce that
$\Phi_\beta\leq\eta\leq\varphi$ on $\p D$.

\item $\Phi_\beta(x)$ is strictly increasing
with respect to $\beta$ and
\begin{equation}\label{eqn.Phib}
\lim_{\beta\ra+\infty}\Phi_\beta(x)=+\infty,
~\forall r_A(x)\geq 1,
\end{equation}
by the definition of $\Phi_\beta(x)$
and \lref{lem.psi}-\textsl{(ii)}.

\item As we showed in \eref{eqn.phi-mu} and \eref{eqn.phi-O},
for any $\beta\geq 1$, we have
\begin{eqnarray*}
\Phi_\beta(x)
&=&\eta+\int_{\bar{r}}^{r_A(x)}\tau \psi(\tau,\beta)d\tau\\
&=&\eta+\frac{1}{2}(r_A(x)^2-\bar{r}^2)
+\int_{\bar{r}}^{r_A(x)}\tau\big(\psi(\tau,\beta)-1\big)d\tau\\
&=&\frac{1}{2}r_A(x)^2+\left(\eta-\frac{1}{2}\bar{r}^2
+\mu_{\bar{r}}(\beta)\right)-\mu_{r_A(x)}(\beta)\\
&=&\frac{1}{2}r_A(x)^2+\mu(\beta)-\mu_{r_A(x)}(\beta)\\
&=&\frac{1}{2}x^TAx+\mu(\beta)+O\left(|x|^{-m+2}\right)
~(|x|\ra+\infty),
\end{eqnarray*}
where we set
\[\mu(\beta):=\eta-\frac{1}{2}\bar{r}^2+\mu_{\bar{r}}(\beta),\]
and used the fact that $x^TAx=O(|x|^2)~(|x|\ra+\infty)$
since $\lambda(A)\in\Gamma^+$.
\end{enumerate}

\medskip

\emph{Step 2.}\quad
For fixed $\hat{r}>\bar{r}$, there exists $\hat{\beta}>1$ such that
\[\min_{\p E_{\hat{r}}}\Phi_{\hat{\beta}}
>\max_{\p E_{\hat{r}}}Q,\]
in light of \eref{eqn.Phib}.
Thus we obtain
\begin{equation}\label{eqn.PhiQ}
\Phi_{\hat{\beta}}>Q\quad\mbox{on}~\p E_{\hat{r}}.
\end{equation}

Let
\[\tilde{c}:=\max\left\{\eta,\mu(\hat{\beta}),\bar{c}\right\},\]
where the $\bar{c}$ comes from \rref{rmk.Qxi}-\textsl{(3)},
and hereafter fix $c\geq\tilde{c}$.

By \lref{lem.psi} and \cref{cor.mub} we deduce that
\[\psi(r,1)\equiv 1\Ra\mu_{\bar{r}}(1)=0
\Ra\mu(1)=\eta-\frac{1}{2}\bar{r}^2<\eta\leq\tilde{c}\leq c,\]
and
\[\lim_{\beta\ra+\infty}\mu_{\bar{r}}(\beta)=+\infty
\Ra\lim_{\beta\ra+\infty}\mu(\beta)=+\infty.\]
On the other hand,
it follows from \cref{cor.mub}
that $\mu(\beta)$ is continuous and
strictly increasing with respect to $\beta$
\big{(}which indicates that the inverse of $\mu(\beta)$ exists
and $\mu^{-1}$ is strictly increasing\big{)}.
Thus there exists a unique $\beta(c)$ such that $\mu(\beta(c))=c$.
Then we have
\[\Phi_{\beta(c)}(x)=\frac{1}{2}r_A(x)^2+c-\mu_{r_A(x)}(\beta(c))
=\frac{1}{2}x^TAx+c+O\left(|x|^{-m+2}\right)~(|x|\ra+\infty),\]
and
\[\beta(c)=\mu^{-1}(c)\geq\mu^{-1}(\tilde{c})\geq\hat{\beta}.\]
Invoking the monotonicity of $\Phi_\beta$
with respect to $\beta$ and \eref{eqn.PhiQ}, we obtain
\begin{equation}\label{eqn.PhigQ}
\Phi_{\beta(c)}\geq\Phi_{\hat{\beta}}>Q\quad\mbox{on}
~\p E_{\hat{r}}.
\end{equation}
Note that we already know
\[\Phi_{\beta(c)}\leq Q\quad\mbox{on}
~\ol{E_{\bar{r}}}\setminus D,\]
from (4) of \emph{Step 1}.

\medskip

Let
\begin{equation*}
\ul{u}(x):=
\begin{cases}
\max\left\{\Phi_{\beta(c)}(x),Q(x)\right\},& x\in E_{\hat{r}}\setminus D,\\
\Phi_{\beta(c)}(x),& x\in \R^n\setminus E_{\hat{r}}.
\end{cases}
\end{equation*}
Then we have
\begin{enumerate}[\quad(1)]
\item $\ul{u}$ is continuous and satisfies
\[\frac{\sigma_k(\lambda(D^2{\ul{u}}))}
{\sigma_l(\lambda(D^2{\ul{u}}))}\geq1\]
in $\R^n\setminus\ol{D}$ in the viscosity sense,
by (1) and (3) of \emph{Step 1}.

\item $\ul{u}=Q=\varphi$ on $\p D$,
by (2) of \emph{Step 1}.

\item If $r_A(x)$ is large enough, then
\[\ul{u}(x)=\Phi_{\beta(c)}(x)
=\frac{1}{2}x^TAx+c+O\left(|x|^{-m+2}\right)~(|x|\ra+\infty).\]
\end{enumerate}

\medskip

\emph{Step 3.}\quad
Let
\[\ol{u}(x):=\frac{1}{2}x^TAx+c,
~\forall x\in\R^n.\]
Then $\ol{u}$ is obviously a supersolution and
\[\lim_{|x|\ra+\infty}\left(\ul{u}-\ol{u}\right)(x)=0.\]


To use the Perron's method to establish \lref{lem.hqe},
we now only need to prove that
\[\ul{u}\leq \ol{u}\quad\mbox{in}~\R^n\setminus D.\]
In fact, since
\[\mu_{r_A(x)}(\beta)\geq 0,
\quad\forall x\in\R^n\setminus E_1,
~\forall\beta\geq 1,\]
according to \cref{cor.mub}, we have
\begin{equation}\label{eqn.Phiolu}
\Phi_{\beta(c)}(x)=\frac{1}{2}x^TAx+c-\mu_{r_A(x)}(\beta(c))
\leq\frac{1}{2}x^TAx+c=\ol{u}(x),~\forall x\in\R^n\setminus D.
\end{equation}
\Big{(}We remark that this \eref{eqn.Phiolu}
can also be proved by using the comparison principle,
in view of
\[\Phi_{\beta(c)}\leq\eta\leq\tilde{c}\leq c\leq\ol{u}
\quad\mbox{on}~\p D,\]
and
\[\lim_{|x|\ra+\infty}\left(\Phi_{\beta(c)}-\ol{u}\right)(x)=0.\Big{)}\]
On the other hand, for every $\xi\in\p D$, since
\[Q_\xi(x)\leq\frac{1}{2}x^TAx+\bar{c}
\leq\frac{1}{2}x^TAx+\tilde{c}
\leq\frac{1}{2}x^TAx+c=\ol{u}(x),
~\forall x\in\p D,\]
and
\[Q_\xi\leq Q<\Phi_{\beta(c)}\leq\ol{u}
\quad\mbox{on}~\p E_{\hat{r}}\]
follows from \eref{eqn.PhigQ} and \eref{eqn.Phiolu},
we obtain
\[Q_\xi\leq\ol{u}\quad\mbox{on}
~\p\left(E_{\hat{r}}\setminus D\right).\]
In view of
\[\frac{\sigma_k(\lambda(D^2{Q_\xi}))}{\sigma_l(\lambda(D^2{Q_\xi}))}
=1=\frac{\sigma_k(\lambda(D^2{\ol{u}}))}
{\sigma_l(\lambda(D^2{\ol{u}}))}
\quad\mbox{in}~E_{\hat{r}}\setminus D,\]
we deduce from the comparison principle that
\[Q_\xi\leq\ol{u}\quad\mbox{in}~E_{\hat{r}}\setminus D.\]
Hence
\begin{equation}\label{eqn.Qu}
Q\leq\ol{u}\quad\mbox{in}~E_{\hat{r}}\setminus D.
\end{equation}
Combining \eref{eqn.Phiolu} and \eref{eqn.Qu},
by the definition of $\ul{u}$, we get
\[\ul{u}\leq \ol{u}\quad\mbox{in}~\R^n\setminus D.\]
This finishes the proof of \lref{lem.hqe}.
\end{proof}

\begin{remark}
To prove \lref{lem.hqe}
we have used above \lref{lem.cp} and \lref{lem.pm}
presented in \ssref{subsec.prelem}.
In fact, one can follow the techniques in \cite{CL03}
(see also \cite{DB11} \cite{Dai11} and \cite{LD12})
instead of \lref{lem.pm} to rewrite the whole proof.
These two kinds of presentation
look a little different but are essentially the same.
\end{remark}


\end{document}